\documentclass[a4paper]{amsart}
\usepackage{amsmath,amsthm,amssymb,hyperref,mathrsfs}
\usepackage{aliascnt}
\usepackage{lmodern}
\usepackage[T1]{fontenc}

\usepackage[textsize=footnotesize]{todonotes}

\usepackage{xcolor}
\definecolor{dblue}{rgb}{0,0,0.70}
\hypersetup{
	unicode=true,
	colorlinks=true,
	citecolor=dblue,
	linkcolor=dblue,
	anchorcolor=dblue
}

\makeatletter
\expandafter\g@addto@macro\csname th@plain\endcsname{%
		\thm@notefont{\bfseries}
	}%
\expandafter\g@addto@macro\csname th@remark\endcsname{%
		\thm@headfont{\bfseries}
	}%
\makeatother

\newtheorem{theorem}
{Theorem}[section]	
\newtheorem*{theorem*}{Theorem}

\newaliascnt{lemma}{theorem}

\newtheorem{claim}[theorem]{Claim}
\aliascntresetthe{lemma}
\newtheorem*{lemma*}{Lemma}

\newaliascnt{proposition}{theorem}
\newtheorem{proposition}[proposition]{Proposition}
\aliascntresetthe{proposition}

\newaliascnt{corollary}{theorem}
\newtheorem{corollary}[corollary]{Corollary}
\aliascntresetthe{corollary}

\newtheorem{Kthm}{Theorem}

\theoremstyle{remark}

\newaliascnt{remark}{theorem}

\aliascntresetthe{remark}
\newaliascnt{question}{theorem}
\newtheorem{question}[question]{Question}
\aliascntresetthe{question}
\newaliascnt{conjecture}{theorem}

\aliascntresetthe{conjecture}

\newtheorem*{question*}{Question}

\newaliascnt{definition}{theorem}
\newtheorem{definition}[definition]{Definition}
\aliascntresetthe{definition}

\newaliascnt{example}{theorem}

\aliascntresetthe{example}

\renewcommand{\restriction}{\mathbin\upharpoonright}

\newcommand{\axiom}[1]{\mathsf{#1}} 
\newcommand{\ZFC}{\axiom{ZFC}}
\newcommand{\AC}{\axiom{AC}}
\newcommand{\WO}{\mathrm{WO}}
\newcommand{\CH}{\axiom{CH}}

\newcommand{\DC}{\axiom{DC}}
\newcommand{\ZF}{\axiom{ZF}}

\newcommand{\HOD}{\mathrm{HOD}}
\newcommand{\GCH}{\axiom{GCH}}

\newcommand{\HS}{\axiom{HS}}

\DeclareMathOperator{\dom}{dom}
\DeclareMathOperator{\rng}{rng}
\DeclareMathOperator{\supp}{supp}
\DeclareMathOperator{\rank}{rank}
\DeclareMathOperator{\sym}{sym}

\DeclareMathOperator{\fix}{fix}

\DeclareMathOperator{\id}{id}
\DeclareMathOperator{\aut}{Aut}

\DeclareMathOperator{\Add}{Add}

\newcommand{\forces}{\mathrel{\Vdash}}

\newcommand{\power}{\mathcal{P}}

\newcommand{\PP}{\mathbb P}
\newcommand{\QQ}{\mathbb Q}
\newcommand{\RR}{\mathbb R}

\newcommand{\cB}{\mathcal B}

\newcommand{\cJ}{\mathcal J}
\newcommand{\cK}{\mathcal K}

\newcommand{\sF}{\mathscr F}
\newcommand{\sG}{\mathscr G}

\newcommand{\fc}{\mathfrak c}

\newcommand{\tup}[1]{\langle#1\rangle}

\author{Asaf Karagila}
\thanks{The author was supported by the Royal Society grant no.~NF170989.}
\email[Asaf Karagila]{karagila@math.huji.ac.il}
\urladdr{http://karagila.org}
\address{School of Mathematics,
University of East Anglia.
Norwich, NR4~7TJ, UK
}
\date{December 9, 2019}

\title[Realizing realizability results]{Realizing realizability results\\ with classical constructions}
\begin{document}
\begin{abstract}J.L.~Krivine developed a new method based on realizability to construct models of set theory where the axiom of choice fails. We attempt to recreate his results in classical settings, i.e.\ symmetric extensions. We also provide a new condition for preserving well-ordered, and other particular type of choice, in the general settings of symmetric extensions.
\end{abstract}
\maketitle
\section{Introduction}
Kleene's realizability was developed to study proofs, especially in constructive and intuitionistic settings, and was later adopted by computer scientists as well. Jean-Louis Krivine developed this framework to accommodate classical logic, and created a framework for proving new independence results in set theory (see \cite{Krivine:RII} and \cite{Krivine:RIII} for details). It is unclear whether or not this construction is a novel way to present forcing-based construction, or if it is truly a new tool, in which case is it somehow equivalent to classical constructions?

In this paper we aim to try to shed some light on this topic by providing two constructions of Krivine by classical means, a forcing-based approach using the technique of symmetric extensions. We also include a discussion on a third model, and why the same approach as the others failed us, what could be done to solve it. After the first version of this paper was written, Krivine announced a result implying that this third model is equivalent to a symmetric extension of a model of $\ZFC$ (see \autoref{subsect:failure}).

While corresponding with Krivine, he informed us that together with Laura Fontanella they proved some weak versions of the axiom of choice, specifically well-ordered choice, in some of the models constructed by Krivine, but also choice from families indexed by some of the ``paradoxical sets'' added to the model. Another recent work,\cite{FontanellaGeoffroy}, by Fontanella and Guillaume Geoffroy is concerned with producing realizability models where $\DC_\kappa$ holds for some uncountable ordinal $\kappa$. 

In \autoref{sect:pres} we prove a structural theorem for symmetric extensions that provides a condition for preserving the axiom of choice from well-ordered families in a symmetric extension, this is done by generalizing the proof of Theorem~8.9 in \cite{Jech:AC1973} where something similar is proved in the context of $\ZF$ with atoms. Our theorem is general enough to accommodate the axiom of choice from the ``paradoxical sets'' mentioned to us by Krivine, and indeed it applies to our constructions. In addition to that, we use other structural theorems for symmetric extensions to show that Krivine's results can be slightly modified to obtain $\DC_\kappa$ for any fixed $\kappa$.

We should perhaps clarify that the specific theorems of Krivine which we discuss are not important, nor are they particularly interesting. What is interesting is the fact that realizability can be used to construct models of $\ZF$. The work here is focused on reproving the same results in hope that we can start building a bridge of understanding between classical methods and realizability methods, and in particular to help and understand if realizability models are in some sense in correspondence with symmetric extensions.
\subsection{Acknowledgements}
The author would like to thank Jean-Louis Krivine for his kind help explaining some points regarding his construction, and to Jonathan Kirby for a helpful discussion about model theoretic properties related to the third model. The author would also like to thank David Schrittesser and Yair Hayut for their helpful suggestions in correcting some of the problems in previous versions of this manuscript. And finally, we want to express our gratitude to the referee whose suggestions helped improve the exposition and readability of this paper.
\section{Preliminaries}
We use $|X|$ to denote the cardinality of $X$, which is the least ordinal equipotent with $X$ if such exists, or the Scott cardinal of $X$, namely \[\{Y\in V_\alpha\mid\exists f\colon Y\to X\text{ a bijection}\}\] where $\alpha$ is the least ordinal for which this set is non-empty. We say a cardinal is an $\aleph$ if it is the cardinal of an infinite ordinal.

We will write $|X|\leq|Y|$ if there is an injection from $X$ into $Y$, and we write $|X|\leq^*|Y|$ if $X$ is empty or there is a surjection from $Y$ onto $X$.\footnote{Equivalently, $|X|\leq^*|Y|$ if there is a subset $Y'\subseteq Y$ and surjection from $Y'$ onto $X$.} This is a reflexive and transitive relation on the cardinals, but it is not necessarily antisymmetric, therefore the meaning of $|X|<^*|Y|$ is $|X|\leq^*|Y|$ and $|Y|\nleq^*|X|$. See the well-written \cite{BanasMoore:1990} for additional information on $<^*$.

For a set $X$ we denote by $\AC_X$ the statement that every family of non-empty sets indexed by $X$ admits a choice function. We will denote by $\AC_\WO$ the statement that for every $\aleph$-number, $\kappa$, $\AC_\kappa$ holds, and $\AC$ will denote $\forall X\,\AC_X$. Dependent Choice for $\kappa$, or $\DC_\kappa$, is the statement that every tree $T$ which is $\kappa$-closed has a maximal element or a chain of type $\kappa$. We write $\DC_{<\kappa}$ to denote $\forall\lambda<\kappa, \DC_\lambda$, and $\DC$ to mean $\DC_{\aleph_0}$. Chapter~8 of \cite{Jech:AC1973} contains numerous theorems and independence results on $\AC_\WO$ and $\DC_\kappa$. For example, $\AC_\WO$ implies $\DC$, but not $\DC_{\omega_1}$.

We follow the standard practices regarding forcing. If $\PP$ is a notion of forcing, then it is a preordered set with a maximum $1$ whose elements are called conditions, and we write $q\leq p$ to denote that $q$ \textit{extends} $p$ or that $q$ is a \textit{stronger condition} than $p$. Two conditions are compatible if they have a common extension, and otherwise they are incompatible.

If $\{\dot x_i\mid i\in I\}$ is a collection of $\PP$-names, we use $\{\dot x_i\mid i\in I\}^\bullet$ to denote the name $\{\tup{1,\dot x_i}\mid i\in I\}$. This notation extends to other forms of ``canonical definitions'' such as ordered pairs or tuples in the obvious way. Note that using this notation, the canonical names for ground model sets can be written as $\check x=\{\check y\mid y\in x\}^\bullet$.

If $\dot x$ and $\dot y$ are $\PP$-names, we say that $\dot y$ \textit{appears in $\dot x$} if there is a condition $p$ such that $\tup{p,\dot y}\in\dot x$. And for a condition $p$ and a $\PP$-name $\dot x$, we write $\dot x\restriction p$ to denote the name $\{\tup{q,\dot y\restriction p}\mid q\leq p, q\forces\dot y\in\dot x,\text{ and }\dot y\text{ appears in }\dot x\}$. It is easy to verify that $p\forces\dot x\restriction p=\dot x$, and if $q$ is incompatible with $p$, then $q\forces\dot x\restriction p=\check\varnothing$.

We write $\Add(\omega,X)$ to denote the partial order whose conditions are finite partial functions $p\colon X\times\omega\to 2$ ordered by reverse inclusion. For a condition $p$ in $\Add(\omega,X)$, we write $\supp(p)$ as the projection of $\dom p$ to $X$.
\subsection{Symmetric extensions}
Let $\PP$ be a notion of forcing, and let $\pi$ be an automorphism of $\PP$. We can extended $\pi$ to act on $\PP$-names. This action is defined recursively, \[\pi\dot x=\{\tup{\pi p,\pi\dot y}\mid\tup{p,\dot y}\in\dot x\}.\]
If $\sG$ is a subgroup of $\aut(\PP)$ we denote by $\sym_\sG(\dot x)$ the group $\{\pi\in\sG\mid\pi\dot x=\dot x\}$.

Let $\sG$ be a group, we say that $\sF$ is a \textit{normal filter of subgroups} over $\sG$ if $\sF$ is a non-empty family of subgroups of $\sG$ which is closed under supergroups and intersection, and for all $H\in\sF$ and $\pi\in\sG$, $\pi H\pi^{-1}\in\sF$.

A \textit{symmetric system} is a triplet $\tup{\PP,\sG,\sF}$ such that $\PP$ is a notion of forcing, $\sG$ is a subgroup of $\aut(\PP)$, and $\sF$ is a normal filter of subgroups over $\sG$.\footnote{We can relax this to require that $\sF$ is a normal filter base.} For the remaining discussion we fix a symmetric system.

We call a $\PP$-name, $\dot x$, \textit{$\sF$-symmetric} if $\sym_\sG(\dot x)\in\sF$. We say that $\dot x$ is hereditarily $\sF$-symmetric if this condition holds hereditarily for every name appearing in $\dot x$. The class of hereditarily $\sF$-symmetric names is denoted by $\HS_\sF$.

\begin{lemma*}[Lemma~14.37 in \cite{Jech:ST2003}]
Let $\pi\in\aut(\PP)$, let $\dot x$ be a $\PP$-name, and let $\varphi(x)$ a formula in the language of set theory.\[p\forces\varphi(\dot x)\iff\pi p\forces\varphi(\pi\dot x).\]
\end{lemma*}
\begin{theorem*}[Lemma~15.51 in \cite{Jech:ST2003}]
Suppose that $G\subseteq\PP$ is a $V$-generic filter, and let $M$ denote $\HS_\sF^G=\{\dot x^G\mid\dot x\in\HS_\sF\}$. Then $V\subseteq M\subseteq V[G]$, and $M$ is a transitive class model of $\ZF$ in $V[G]$.
\end{theorem*}
We call the class $M$ in the theorem above a \textit{symmetric extension of $V$}. These are models where the axiom of choice may fail, and they are one of the main tools for proving independence results related to the axiom of choice.

Arguments about symmetric extensions have their own forcing relation $\forces^\HS$, which is simply described as the relativization of the forcing relation to the class $\HS$. This relation has a forcing theorem, namely $p\forces^\HS\varphi$ if and only if there is a $V$-generic filter $G$ such that $p\in G$, and $\HS^G\models\varphi$.

We say that a condition $p$ is \textit{$\sF$-tenacious} if there is a group $H\in\sF$ such that for all $\pi\in H$, $\pi p=p$. We say that $\PP$ is $\sF$-tenacious if it has a dense set of $\sF$-tenacious conditions. This notion is useful when we want to assume that some fixed conditions are not moved by any of our relevant automorphisms. It turns out that every symmetric system is equivalent to one in which all the conditions are tenacious, see \S12 in \cite{Karagila:2019} for details, which is why we can always assume without loss of generality that our system is tenacious.

As the symmetric systems will always be clear from the context we will omit the subscripts to improve the readability of the text.
\section{Preserving bits of choice}\label{sect:pres}
Sometimes we only wish to show that a certain assumption does not imply the axiom of choice, but we are especially interested in preserving \textit{some} weak choice principles. In \cite{Karagila:2018c} we study some properties which lets us preserve $\DC$, but we are interested in more. Here we will prove that a broad class of $\AC_X$, and in particular $\AC_\WO$, can be preserved assuming certain conditions on the symmetric system. These conditions are somewhat contrived, but hold naturally in the standard cases where these lemmas apply.

Our goal is to generalize an argument that was used to preserve $\AC_\WO$ in the proof of Theorem~8.9 in \cite{Jech:AC1973} where it is shown that if $\kappa$ is an infinite cardinal, then $\DC_{\kappa^+}$ does not follow from $\AC_\WO$ in the context of $\ZF$ with atoms. The idea is to add $\kappa^+$ new sets and take permutations of them with the filter generated by pointwise stabilizers of sets of size $\kappa$. Then, given a family of non-empty sets which is well-ordered, we pick a set $E$ of size $\kappa$ so that when we pick an arbitrary object in a fixed set in the family, we can argue that we can assume without loss of generality that it is fixed by permutations fixing $E$ pointwise. This lets us uniformly choose representatives, and therefore provides a choice function.

We will need a handful of definitions to simplify the statement of the theorem.

\begin{definition}
We say that a symmetric system $\tup{\PP,\sG,\sF}$ is \textit{$\kappa$-mixable} if whenever $\dot x_\alpha\in\HS$ for $\alpha<\gamma<\kappa$, and $\{p_\alpha\mid\alpha<\gamma\}$ is an antichain, then there is $\dot x\in\HS$ such that for all $\alpha$, $p_\alpha\forces\dot x_\alpha=\dot x$. If $\PP$ is $\kappa$-c.c., we simply say that it is \textit{mixable}.
\end{definition}
The immediate examples of mixable symmetric systems, which (by sheer coincidence) are those we use in this paper, are those where the chain condition of $\PP$ is less or equal than the completeness of $\sF$. In that case, we can simply intersect all the groups of the names being mixed. It is proved in \cite{Karagila:2018c} that if $\PP$ is $\kappa$-c.c.\ and $\sF$ is $\kappa$-complete, $\DC_{<\kappa}$ holds as well, indeed this is the proof of Lemma~3.3 in the paper.

\begin{definition}
We say that $\sF$ is an \textit{almost uniform filter} if there is some $H\in\sF$ such that for all $H_0,H_1\in\sF$, there is some $\pi\in H_0$ such that $H\cap H_0\subseteq\pi H_1\pi^{-1}$, we say that $H$ is an \textit{absolute representative of (an almost uniform filter) $\sF$} if we can replace $H$ by any of its conjugates. If $\sF$ is an almost uniform filter of subgroups, we say that the symmetric system is almost uniform, and similarly if $\sF$ admits an absolute representative.
\end{definition}

We say that $\dot X=\{\dot x_i\mid i\in I\}^\bullet$ is an \textit{injective name} when $1\forces\dot x_i\neq\dot x_j$ whenever $i\neq j$, and we say that $H\in\sF$ \textit{measures} $\dot X$ if for all $i\in I$, $H\subseteq\sym(\dot x_i)$ or $H\cup\sym(\dot x_i)$ generates $\sG$. If such $H$ exists, we say that $\dot X$ is a \textit{measurable} name. Finally, $\dot X$ is \textit{densely measurable}\footnote{In a previous version of the manuscript this was called universally measurable, but Sandra M\"uller remarked that this name might cause confusion, and Yair Hayut suggested the current name.} if for every $H\in\sF$ there is $K\subseteq H$ such that $K$ measures $\dot X$.
\begin{theorem}\label{thm:choice-lemma}
Let $\tup{\PP,\sG,\sF}$ be a mixable symmetric system admitting an absolute representative. If $\dot X\in\HS$ is an injective and densely measurable name, then $1\forces^\HS\AC_{\dot X}$. 
\end{theorem}
\begin{proof}
Suppose that $\dot F\in\HS$ and $1\forces^\HS``\dot F$ is a function with domain $\dot X$, and for all $i\in I, \dot F(\dot x_i)\neq\check\varnothing$''.\footnote{Full generality requires that we assume the condition is $p$, rather than $1$. But we can either work ``below $p$'' or use the mixability to replace $\dot F$ with another name that would be equivalent to $\dot F$ below $p$.} Let $K\in\sF$ denote $\sym(\dot X)\cap\sym(\dot F)$. For each $i$ let $K_i$ denote $K\cap\sym(\dot x_i)$. Finally, fix an absolute representative $H$ which measures $\dot X$.

For every $i$, define $\dot A_i$ as follows, \[\dot A_i=\left\{\tup{p,\dot a}\mid p\forces\dot a\in\dot F(\dot x_i), \rank(\dot a)<\rank(\dot F),\text{ and }\dot a\in\HS\right\},\] where $\rank$ denotes the rank of a name.\footnote{Any reasonable rank, e.g.\ the von Neumann rank, would work (mutantis mutandi), though.}
\begin{claim}
For all $i\in I$, $1\forces\dot A_i=\dot F(\dot x_i)$. If $\pi\in K$ and $\pi\dot x_i=\dot x_j$, then $\pi\dot A_i=\dot A_j$. In particular, $K_i\subseteq\sym(\dot A_i)$.
\end{claim}
\begin{proof}[Proof of Claim]\renewcommand{\qedsymbol}{$\square$ (Claim)}
Clearly $1\forces\dot A_i\subseteq\dot F(\dot x_i)$. In the other direction, suppose that $p\forces\dot x\in\dot F(\dot x_i)$, then there is some $q\leq p$ and $\dot a\in\HS$ such that $\rank(\dot a)<\rank(\dot F)$, and $q\forces\dot a=\dot x$. Therefore $\tup{q,\dot a}\in\dot A_i$, so $q\forces\dot x\in\dot A_i$, and therefore $1\forces\dot F(\dot x_i)\subseteq\dot A_i$.

Now suppose that $\pi\in K$, then since $\pi\dot F=\dot F$ and $\pi\dot x_i=\dot x_j$, we get that $p\forces\dot a\in\dot F(\dot x_i)$ if and only if $\pi p\forces\pi\dot a\in\dot F(\dot x_j)$. Since rank is preserved under automorphisms, it shows that $\pi\dot A_i=\dot A_j$. In particular, if $\pi\dot x_i=\dot x_i$ (i.e.\ $\pi\in K_i$), then $\pi\dot A_i=\dot A_i$.
\end{proof}
Since we also only take $\dot a$ which are in $\HS$ in the definition of $\dot A_i$, we have that $\dot A_i\in\HS$, and it is a ``semi-canonical'' name for $\dot F(\dot x_i)$.

For every $i$, let $\dot a_i\in\HS$ be a name, obtained using the mixability of the symmetric system, such that $1\forces\dot a_i\in\dot A_i$. We can assume that for all $i$, $K_i\cap H\subseteq\sym(\dot a_i)$, otherwise pick some $\pi\in K_i$ such that $K_i\cap H\subseteq\pi\sym(\dot a_i)\pi^{-1}=\sym(\pi\dot a_i)$. But since $\pi\dot A_i=\dot A_i$, we have $1\forces\pi\dot a_i\in\dot A_i$, so we can take it instead.

Finally, consider the orbits of the $\dot x_i$'s under the group $H$, and let $J\subseteq I$ be such that $\{\dot x_j\mid j\in J\}$ is a system of representatives from each orbit. And let $\dot f=\{\tup{\pi\dot x_j,\pi\dot a_j}^\bullet\mid j\in J,\pi\in H\}^\bullet$. We claim that $1\forces\dot f(\dot x_i)\in A_i=\dot F(\dot x_i)$ for all $i$. To see that it is defined on all $\dot X$, simply note that $J$ is a system of representatives for the orbits under $H$, and to see that $\dot f$ indeed is a name for a function, suppose that $\pi\dot x_i=\dot x_i$, then $\pi\in K_i\cap H$, but for some $j\in J$ and $\sigma\in H$ we have $\sigma\dot x_j=\dot x_i$ and $\sigma\dot a_j=\dot a_i$. Write $\pi=\sigma\widehat{\pi}\sigma^{-1}$, where $\widehat{\pi}\in K_j\cap H$, then $\widehat{\pi}\dot a_j=\dot a_j$, and therefore $\pi\dot a_i=\dot a_i$. This completes the proof that $\dot f$ is a choice function from $\dot F$.
\end{proof}
\begin{corollary}\label{cor:ACWO}
Suppose that $\tup{\PP,\sG,\sF}$ satisfies the assumptions of \autoref{thm:choice-lemma}, then $1\forces^\HS\AC_\WO$.
\end{corollary}
\begin{proof}
Note that for any ordinal $\eta$, $\check\eta$ is injective and measured by $\sG$.
\end{proof}
Note that this gives us a different proof of $\DC$ in the case of a c.c.c.\ forcing with a $\sigma$-complete filter of subgroups, since $\AC_\WO$ implies $\DC$ (this is Theorem~8.2 in \cite{Jech:AC1973}, originally due to Jensen). However, this does not extend to $\DC_\kappa$ for any uncountable $\kappa$, since $\AC_\WO$ is not sufficient to prove those over $\ZF$.
\begin{question}
What are the exact assumptions we need to make in order to preserve $\AC_{\dot X}$ in general?
\end{question}

It is a good place as any to point out that Solovay's model is constructed by a mixable system without an absolute representative, and indeed $\AC_{\aleph_1}$ fails there.
\section{Classical approach to new results}
Krivine's results use the method of realizability to create new models of $\ZF+\DC$ where there are some sets of real numbers with particular properties that reflect peculiarities in the cardinal structure below $2^{\aleph_0}$.

The intuition behind creating structures in models where $\AC$ fails comes from the principle ``if you want it, preserve it''. A plethora of examples arise from just adding countably many Cohen reals---which is the same as adding a single Cohen real---and creating different symmetric extensions to preserve different kinds of structures. The question is always how to naturally present the forcing so that we can find a reasonable group of automorphism acting on it, and what filter of subgroups we use to preserve bits and pieces of it.

In the simplest case we start with a set $X$ and force with $\Add(\omega,X)$. The permutations of $X$ act naturally on the forcing by $\pi p(\pi x,n)=p(x,n)$. If $X$ is assumed to have some additional structure (e.g.\ a group structure) and we take automorphisms of that structure, then we will preserve the structure on the generic copy of $X$, even if the generic copy of $X$, and thus the set of real numbers, is not well-orderable in the symmetric extension.

Classically, we are often not interested in the reals themselves, but rather the structure. Which means that we normally take $X\times\omega$ or $X\times\kappa$, rather than $X$. The reason is simple: the real numbers are linearly ordered. If we want to control the subsets of the copy of $X$ in the symmetric extensions, it helps when there are none added by the linear ordering. So by adding an infinite set of Cohen reals for each $x$ in $X$, we ensure that those are sufficiently indiscernible to prevent any set theoretic definability issues introducing unwanted subsets of the generic copy of $X$ into the symmetric extension (e.g.\ the set of those reals which have a certain initial segment). Here, however, we care less about the structure's subsets. So using the real numbers directly is not a matter of concern.

All our systems will satisfy that $|X|$ is a successor cardinal, and the filter of groups are generated by sets of smaller cardinality. Therefore the conditions of having an absolute representative and mixability are immediate to verify. Finally, by the definition of the action, it will be clear that if $\dot X$ is the canonical name for the Cohen reals added, then it is certainly injective and densely measurable.
\subsection{Prelude to Model I}
Krivine's first result from \cite{Krivine:RII} is as follows.
\begin{Kthm}[Krivine, Theorem~5.5 in \cite{Krivine:RII}]
It is consistent with $\ZF+\DC$ that there is a sequence of sets $A_n\subseteq\RR$ for $n<\omega$, such that\begin{enumerate}
\item for $n>1$ $A_n$ is uncountable, 
\item $|A_n|<|A_m|$ if and only if $|A_n|<^*|A_m|$ if and only if $n<m$, and 
\item $|A_n\times A_m|=|A_{nm}|$.
\end{enumerate} 
\end{Kthm}

We start this proof with an outline of a construction that is doomed to fail. But from our failure we will make the approach that does work clearer, rather than the usual ``how did you even come up with this idea?'' which sometimes plague mathematical constructions.\footnote{Other than clarifying the train of thoughts of the author, we would like to reinforce the view that ``if you never fall, you will never learn to get up'' which is something many young researchers might struggle with.}

In his paper Krivine uses sets of the form $\gimel\mathbf n$ for $n<\omega$. We do not understand these sets. If, however, we intuitively think about them as somehow being ``reasonable names of elements of $n$'', then they can be seen in some sense as an ultrapower of the natural numbers. This will produce an uncountable set and will obey the arithmetic requirements of the theorem. We want to stress that this is \textit{not quite} the right intuition, but in some sense Theorem~4.20 in \cite{Krivine:RII} makes this seem like a plausible intuition to start with. Especially when thinking of Krivine's ``$x\mathrel{\varepsilon}\gimel\mathbf{n}$'' as analogous to ``$1\forces\dot x\in\check n$''. 

Of course, this set of names is in the ground model, so it could not possibly be the set we are looking for, since it can be well-ordered. But if we work under the assumption that the real numbers added in Krivine's work are anything like Cohen reals, we can instead add many Cohen reals and then look at functions that give us the pointwise interpretation of the names. This is still not enough, we need to require that these interpretations are symmetric with respect to whatever symmetric system we use. Of course, realizability models are not the same as forcing (or symmetric) extensions, but if $M$ is the ``ground model'' of a realizability construction, then the final outcome has an inner model which is an elementary extension of $M$.\footnote{This reinforces the idea that somehow an ultrapower is involved, although this is just a place to start intuitively thinking about realizability models.} Since being a Cohen real over an inner model is a first-order property in the language of set theory, we can ask this sort of question in realizability models.

The obvious candidates is, as always, Cohen forcing. Since we want to preserve $\AC_\WO$ we need to add some $\kappa^+$ Cohen reals, say $\omega_1$ for concreteness sake. We use the symmetric extension given by $\PP=\Add(\omega,\omega_1)$. The group $\sG$ is the group of all permutations of $\omega_1$ acting naturally on $\PP$. The filter $\sF$ is generated by $\{\fix(E)\mid E\in[\omega_1]^{<\omega_1}\}$, where $\fix(E)=\{\pi\in\sG\mid\pi\restriction E=\id\}$.

Now we have a problem. If we consider the function defined by each of the names in $\gimel\mathbf{n}$ (which we do not define here, see \cite{Krivine:RII} or \cite{Krivine:RIII} for the definition), we have no clear way of coding this sort of sequence into a single real number. It is true that each such name \textit{is} a real number, but it is also in the ground model. So this collection of names is well-orderable. We want instead to code the interpretation of this name by all the canonical Cohen reals simultaneously.

We can instead consider functions which are the interpretation functions on a countable set of our Cohen reals, and are constant outside. It is still unclear that we can code this with a single real number. Specifically, while $\AC_X$ (where $X$ is the set of Cohen reals) holds by \autoref{thm:choice-lemma}, it is not enough to choose for every countable family of Cohen reals an enumeration. Not even if we restrict ourselves to countable sets which come from the ground model (i.e.\ $\{\dot x_\alpha\mid\alpha\in E\}^\bullet$ for some $E\in[\omega_1]^{<\omega_1}$), since we still need to uniformly choose the enumerations even if they do come from the ground model.

It would certainly help if we can identify a reasonable family of countable subsets of $X$ which is both rich enough, and can be uniformly enumerated. But we are also facing a problem when we consider infinite sets as the ``true domain'' of our functions. This can easily lead to coding too many subsets and ending up with the full power set of $X$, which may be larger than the reals in the model, or it might end up being equipotent to its square which would somehow defeat the purpose of this construction to begin with.

\subsection{Model I: a multiplicative sequence of sets}
We are ready to prove Krivine's theorem mentioned above. Instead of just adding $\omega_1$ Cohen reals, we add $\omega_1\times\QQ$, and let $\dot x_{\alpha,q}$ denote the name of the real corresponding to the $\tup{\alpha,q}$th coordinate. Our automorphism group is going to be the order automorphisms of $\omega_1\times\QQ$ with the lexicographic order. It is important to note that this order is very homogeneous. The filter $\sF$ is generated by $\{\fix(E)\mid E\in[\omega_1\times\QQ]^{<\omega_1}\}$. Immediately we obtain that $\DC$ holds, and in fact by \autoref{cor:ACWO} also $\AC_\WO$.

Let $\dot X=\{\dot x_{\alpha,q}\mid \tup{\alpha,q}\in\omega_1\times\QQ\}$, and let $\dot\prec$ denote the order inherited on $\dot X$ which is given by \[\dot\prec=\{\tup{\dot x_{\alpha,q},\dot x_{\alpha',q'}}^\bullet\mid\tup{\alpha,q}<_{\mathrm{lex}}\tup{\alpha',q'}\}^\bullet.\]
As all of our automorphisms are order preserving, it is easy to see that $\dot\prec$ is indeed in $\HS$. Let $G$ be a $V$-generic filter, and let $\cK_1$ denote the symmetric extension $\HS^G$. We omit the dot from the name to denote its interpretation, e.g.\ $\dot X^G$ will be denoted by $X$. It follows from what we saw until now that $X$ has a linear order which is externally isomorphic to $\omega_1\times\QQ$, i.e.\ there is an isomorphism in $V[G]$ but not in the symmetric extension.

\begin{proposition}
In $\cK_1$ every proper initial segment of $\tup{X,\prec}$ is countable, $X$ cannot be mapped onto $\omega_1$, but $X$ is uncountable.
\end{proposition}
\begin{proof}
We start from the end. To see that $X$ is uncountable, note that in $V[G]$ we do have a bijection between $X$ and $\omega_1$, and since we did not collapse $\omega_1$, it is impossible that $X$ is countable $\cK_1$.

Suppose that $p\forces^\HS\dot f\colon\dot X\to\check\omega_1$. Then there is some countable $E\subseteq\omega_1\times\QQ$, without loss of generality an initial segment, such that $\fix(E)$ is a subgroup of $\sym(\dot f)$ and $\supp(p)\subseteq E$. Take $\tup{\alpha,q}$ which is a proper upper bound of $E$ and not its supremum (if it exists). By Cohen forcing having the c.c.c.\ there is some $\beta$ such that $p\forces^\HS\dot f(\dot x_{\alpha,q})<\check\beta$. 

Let $\tup{\alpha',q'}$ another upper bound of $E$ as before, and let $p'\leq p$ be such that $p'\forces^\HS\dot f(\dot x_{\alpha',q'})>\check\beta$, if there are no such $\tup{\alpha',q'}$ and $p'$, then $p$ must force that $\dot f$ is not surjective. If there is such $p'$, consider now $\pi$ to be any order automorphism which moves $\tup{\alpha,q}$ to $\tup{\alpha',q'}$ while not changing any coordinate in $E$. Since $\pi\in\fix(E)$ it follows that $\pi p=p$, and therefore $\pi p'\leq p$ as well. However $\pi p'\forces^\HS\dot f(\dot x_{\alpha,q})>\check\beta$. But since $\pi p'\leq p$ it mean that $\pi p'\forces^\HS\dot f(\dot x_{\alpha,q})<\check\beta$. Therefore there is no $\dot f\in\HS$ such that any $p$ forces $\dot f$ to be a surjection from $\dot X$ onto $\check\omega_1$. Therefore $\cK_1$ satisfies the second property.

A similar proof shows that if $A\subseteq X$ is an interval, then it is bounded if and only if it is countable. In particular, no proper initial segment is uncountable.
\end{proof}

Our sequence of $A_n$'s is going to be derived from a sequence of powers of $X$. As we have no choice in the matter, $A_0=\varnothing$ and $A_1=\{1\}$. We wish to have $A_2=X$, which means that $A_{2^n}=X^n$. This, again, follows in some sense after Krivine's proof where he first embeds all the $\gimel\mathbf{2^n}$, and then use them to derive the embeddings of the rest. We present a more direct approach to the definition of our $A_n$'s.

\begin{definition}
We say that a function $f\colon X\to\omega$ is \textit{based} if it is weakly decreasing and for every $n<\omega$, if $f^{-1}(n)$ is non-empty, then it admits a least element or it is an initial segment of $X$. We call the least element of $f^{-1}(n)$ \textit{the base point of $n$}.
\end{definition}

Working in $\cK_1$, let $A_n$ denote the set $\{f\colon X\to n\mid f\text{ is based}\}$, clearly $A_n\subseteq A_m$ for $n\leq m$. We can code all the $A_n$'s uniformly into the reals, since each based function is determined entirely by its finite set of base points and their values. Moreover, $A_0=\varnothing$, and $A_1$ is a singleton. Note that $A_2$ is in fact a copy of $X$, since a based function into $2$ is simply identifying a point where the value drops from $1$ to $0$ (there are two constant functions, but because of $\DC$ we can freely ignore those).

Note that since a based function is determined by a finite set of points and natural numbers, it is in fact a copy of a based function from $\omega_1\times\QQ$ in the ground model. If $F$ is such a based function in $V$, we let $\dot f_F$ be the name $\{\tup{\dot x_{\alpha,q},\check n}^\bullet\mid f(\alpha,q)=n\}^\bullet$. This name is in $\HS$ since taking its maximal base point to be $\tup{\alpha,i}$, we get that any initial segment which contains it is a support for $\dot f_F$. We can therefore define $\dot A_n$ to be the name \[\{\dot f_F\mid F\colon\omega_1\times\QQ\to\omega\text{ is a based function}\}^\bullet.\] Note that these names satisfy that $\fix(\dot A_n)=\sG$, as $\pi\dot f_F$ is in fact $\dot f_{F\circ\pi}$ which is also based.

\begin{proposition}
For all $n,m<\omega$, $|A_n\times A_m|=|A_{nm}|$.
\end{proposition}
\begin{proof}
Given $f_n\in A_n$ and $f_m\in A_m$ define the function $f_{nm}(x)=m\cdot f_n(x)+f_m(x)$. First we need to verify that $f_{nm}$ is based. Observe that the case $mn=0$ is trivial, since $A_0=\varnothing$, in which case the equality holds for the sets, not just their cardinality.

As a start we show that $f_{nm}$ is weakly decreasing. If $x\leq y$, then $f_m(x)\geq f_m(y)$ and $f_n(x)\geq f_n(y)$. Therefore $m\cdot f_n(x)\geq m\cdot f_n(y)$ and so \[m\cdot f_n(x)+f_m(x)\geq m\cdot f_n(y)+f_m(y).\]

Next we show that it admits base points. Suppose $f_{nm}^{-1}(i)$ is non-empty and write $i=mj+k$ such that $j<n$ and $k<m$. By definition on $f_{nm}$, $f_m^{-1}(k)$ and $f_n^{-1}(j)$ are non-empty as well. Let $x\in X$ be the maximum between the base point of $k$ in $f_m$ and the base point of $j$ in $f_n$. It follows that $f_{nm}(x)=i$, but we also claim it is the minimum point satisfying this. If $y<x$, then either $f_n(x)<f_n(y)$ or $f_m(x)<f_m(y)$ (and weak inequality holds for the other function), which in turn imply that either $m\cdot f_n(x)\leq m\cdot f_n(y)$ and $f_m(x)\leq f_m(y)$ with at least one of these being a strict inequality. Therefore $m\cdot f_n(x)+f_m(x)<m\cdot f_n(y)+f_m(y)$ holds, as wanted.

Finally, $\tup{f_n,f_m}\mapsto f_{nm}$, as defined above, is a bijection, since $\tup{i,j}\mapsto m\cdot i+j$ is a bijection from $n\times m$ to $n\cdot m$, so we can decode the pair $\tup{f_n,f_m}$ from $f_{nm}$.
\end{proof}
The above proposition is quite similar in its nature to Theorem~4.21 in \cite{Krivine:RII}.
\begin{theorem}
$\cK_1\models |A_n|<^*|A_m|$ if and only if $n<m$.
\end{theorem}
To make the proof clearer, we will confuse the $\bullet$-name of a based function (as defined above) and the ground model function which induces it. In particular, if we say that $E\subseteq\omega_1\times\QQ$ and $\dot f$ has its base points in $E$, we mean that the base points of the based function $F$ such that $\dot f_F=\dot f$ (in our previous notation) are inside $E$.
\begin{proof}
Since $A_n\subseteq A_m$ if and only if $n\leq m$, it is enough to prove that there is no surjection from $A_n$ onto $A_m$ when $n<m$. Of course, we may start by assuming that $1<n<m$, since for $n\leq 1$ this is trivial.

Suppose that $\dot F\in\HS$ and $p\forces^\HS\dot F\colon\dot A_n\to\dot A_m$. Let $E$ be an initial segment such that $\pi\in\fix(E)$ satisfies $\pi\dot F=\dot F$ and $\pi p=p$. Note that there are only countably many based functions in $A_m$ whose base points are in $E$, therefore we can find some $\dot f_m$ which appears in $\dot A_m$ such that:
\begin{enumerate}
\item $\dot f_m$ admits $m$ base points.
\item None of the base points of $\dot f_m$ lie inside $E$.
\end{enumerate}

If there is no $p'\leq p$ and $\dot f_n$ appearing in $\dot A_n$ such that $p'\forces^\HS\dot F(\dot f_n)=\dot f_m$, then $p$ forces that $\dot F$ is not surjective. Otherwise, let $p'$ and $\dot f_n$ be such that $p'\forces^\HS\dot F(\dot f_n)=\dot f_m$. Since $n<m$ there is at least one base point of $\dot f_m$ which is not a base point of $\dot f_n$, say $\dot x_{\alpha,q}$. We can find a small enough interval such that moving this base point does change its type relative to the base points of $\dot f_n$. Let $\pi$ be some automorphism which only moves inside that small interval such that:
\begin{enumerate}
\item $\pi\in\fix(E)$,
\item $\pi(\alpha,q)\neq\tup{\alpha,q}$, and
\item $\pi p'$ is compatible with $p'$.
\end{enumerate}
The second condition is easy to achieve since $p'$ has only finite information in this interval, which is isomorphic to $\QQ$ as a linear order. But since $\pi$ does not move any of the base points of $\dot f_n$, we get that $\pi\dot f_n=\dot f_n$. Therefore \[p'\cup\pi p'\forces^\HS\dot F(\dot f_n)=\dot f_m\neq\pi\dot f_m=\dot F(\dot f_n).\] This is impossible, of course, and therefore no such $p'$ and $\dot f_n$ exist. In other words, $\dot F$ cannot possibly be a surjection.
\end{proof}
This completes the proof that $\cK_1$ satisfies the wanted properties. If we assume that $\kappa$ is some uncountable cardinal, such that there is a universal $\kappa$-dense linear ordering, i.e.\ $\eta_\kappa$, then by replacing $\omega_1\times\QQ$ with $\kappa^+\times\eta_\kappa$ the proof translates in a straightforward way, and since we can now use $\fix(E)$ for $|E|<\kappa^+$, rather than countable, we obtain $\DC_\kappa$ rather than just $\DC$.

We therefore have the following theorem.
\begin{theorem}
Assume that $V\models\ZFC+\GCH$. Let $\kappa$ be any infinite cardinal, then there is a cofinality-preserving symmetric extension $\cK_1(\kappa)$ in which the following statements hold:
\begin{enumerate}
\item $\ZF+\DC_\kappa+\AC_\WO$,
\item There is an $\subseteq$-increasing sequence of sets $A_n\subseteq\RR$ such that for $n>1$, $A_n$ is uncountable, $|A_n|<^*|A_m|$ for $n<m$, and $|A_n\times A_m|=|A_{nm}|$.
\end{enumerate}
\end{theorem}
We remark that in realizability models $\gimel 2$ is a Boolean algebra. We suspect that replacing $A_2$ with something that looks like the interval algebra of $X$ (with its special linear order) might be a way to simulate this Boolean algebra here, and then literally defining $A_n$ as the Boolean ultrapower of $n$ by this Boolean algebra. The $\gimel$ function is confusing enough, but we encourage others who are interested in these ideas to pursue a closer investigation of these approaches for the $\gimel$ function.
\subsection{Model II: Boolean algebras with products}
In \cite{Krivine:RIII} two models are presented. The first model is used to prove the following theorem.
\begin{Kthm}[Krivine, Theorem~34 in \cite{Krivine:RIII}]
It is consistent with $\ZF+\DC$ that there is an embedding, $i\mapsto A_i$, of the countable atomless Boolean algebra, $\cB$, into $\mathcal P(\RR)$ satisfying the following properties:
\begin{enumerate}
\item $A_0=\{0\}$, $|\RR|\leq^*|A_1|$, and $A_i$ is uncountable for all $i\neq 0$,
\item $A_{i\land j}=A_i\cap A_j$,
\item $|A_{i\lor j}|=|A_i\times A_j|$, in particular $|A_i|=|A_i\times A_i|$, and
\item $|A_i|\leq^*|\bigcup_{j\in J}A_j|$ if and only if $i\leq j$ for some $j\in J$, for any $J\subseteq\cB$.
\end{enumerate}
\end{Kthm}
For this model we actually make things a bit easier for ourselves, and embed the entire Boolean algebra $\power(\omega)$ with the above properties of the embedding. This is indeed enough, since the countable atomless Boolean algebra has a very concrete embedding into $\power(\omega)$. We therefore revert to the set-operations on this Boolean algebra, rather than abstract Boolean notation.

We use the forcing $\PP=\Add(\omega,\omega\times\omega_1)$ with $\sG$ the group of permutations, $\pi,$ of $\omega\times\omega_1$ for which $\pi(n,\cdot)$ is a permutation of $\{n\}\times\omega_1$. In other words, the group is the full-support product $\prod_{n<\omega}S_{\omega_1}$, acting naturally on $\omega\times\omega_1$. Our filter of subgroups is given by countable supports, as before. Therefore, as above, we will have $\ZF+\DC$ in the symmetric extension, as well as $\AC_\WO$.

For $\tup{n,\alpha}\in\omega\times\omega_1$ we denote by $\dot x_{n,\alpha}$ the name for $\{\tup{p,\check k}\mid p(n,\alpha,k)=1\}$. Let $\dot X_n=\{\dot x_{n,\alpha}\mid\alpha<\omega_1\}^\bullet$. Clearly, each $\dot X_n$ is symmetric, and indeed, the sequence $\tup{\dot X_n\mid n<\omega}^\bullet\in\HS$ as well. For $f\colon\omega\to\omega\times\omega_1$, let $\dot x_f$ denote the name $\tup{\dot x_{f(n)}\mid n<\omega}^\bullet$, and let $\dot X_n^{\omega,V}=\{\dot x_f\mid f\colon\omega\to\{n\}\times\omega_1\}^\bullet$.

Let $G$ be a $V$-generic filter, and let $\cK_2$ denote the symmetric extension. As before, we omit the dots to indicate the interpretation of the names. Working in $\cK_2$ we define for $S\subseteq\omega$ the set $A_S$ as the product: \[A_S=\prod_{n\in\omega}\begin{cases}X_n^{\omega,V} & n\in S\\\{0\} & n\notin S\end{cases},\]
by coding sequences of real numbers as real numbers we can assume each $X_n^{\omega,V}$ is a set of real numbers, and by applying the coding again we can assume that $A_S\subseteq\RR$ for all $S\subseteq\omega$. We will assume that the constant sequence $0$ will be coded as the number $0$.

We claim that $S\mapsto A_S$ is the wanted embedding. We prove each property in a separate proposition.

\begin{proposition}\label{prop:idems}
$|A_S\times A_S|=|A_S|$.
\end{proposition}
\begin{proof}
It is enough to prove that $|X_n^{\omega,V}\times X_n^{\omega,V}|=|X_n^{\omega,V}|$. If we do that, then by choosing a bijection for each $n\in S$ we get the wanted result. But this is trivial, as the interleaving function, mapping $\tup{f,g}$ to $h$ such that $h(2n)=f(n)$ and $h(2n+1)=g(n)$ for all $n<\omega$, is such a bijection, lifted from the ground model.
\end{proof}
\begin{corollary}
$|A_S\times A_T|=|A_{S\cup T}|$.
\end{corollary}
\begin{proof}
If $S\cap T=\varnothing$ this is trivial. In the general case, note that $A_S\times A_T$ is naturally isomorphic to $(A_{S\setminus T}\times A_{S\cap T})\times (A_{T\setminus S}\times A_{S\cap T})$. Since $A_{S\cap T}\times A_{S\cap T}$ is the same cardinality as $A_{S\cap T}$ the result follows.
\end{proof}

\begin{proposition}
$A_\varnothing=\{0\}$, $|\RR|\leq^*|A_\omega|$, and for all $S\neq\varnothing$, $A_S$ is uncountable.
\end{proposition}
\begin{proof}
The first and third part are immediate from the definition of $A_S$. The fact that $|\RR|\leq^*|A_\omega|$ follows from the fact that every real number is the interpretation of an $\Add(\omega,\omega)$-name using a sequence in reals coded by an element of $A_\omega$.

In the ground model there are only $\fc^V$ names for reals, so it is enough to choose one countable set, e.g. $\{x_{0,n}\mid n<\omega\}$. The collection of sequences in $A_S$ which enumerate this specific set has cardinality $\fc^V$, as those are all enumerations from the ground model. This, together with \autoref{prop:idems} implies that $|A_\omega\times\fc^V|=|A_\omega|$. Enumerate the nice names\footnote{For any reasonable definition of ``nice name''.} of reals in $\Add(\omega,\omega)$ from $V$, then map the pair $(x,\alpha)$ to the interpretation of the $\alpha$th name by the generic coded by $x$.
\end{proof}

\begin{proposition}
$A_S\cap A_T=A_{S\cap T}$.\qed
\end{proposition}

\begin{proposition}
$|A_S|\leq^*|\bigcup_{T\in\cJ}A_T|$ if and only if $S\subseteq T\in\cJ$.
\end{proposition}
\begin{proof}
It is clear that if $S\subseteq T\in\cJ$ then $|A_S|\leq^*|\bigcup_{T\in\cJ}A_T|$. We will show that if $S\nsubseteq T$ for all $T\in\cJ$, then this is not the case. We start with the case where $S,\cJ\in V$, as it simplifies the proof.

An element of $A_S$ is a sequence of sequences which are ``kind of coded by ground model reals''. As such, it has a fairly canonical name given by $\tup{\dot x_{f_n}\mid n\in S}^\bullet$, where $f_n\colon\omega\to\{n\}\times\omega_1$. This provides us with a $\bullet$-name for $A_S$, which we will denote by $\dot A_S$.\footnote{We tacitly ignore the sequences outside of $S$ which are constant $0$ anyway.}

Suppose that $\dot F\in\HS$ and $p\forces^\HS\dot F\colon\bigcup_{T\in\cJ}\dot A_T\to\dot A_S$. Let $E$ be a countable set such that $\fix(E)\subseteq\sym(\dot F)$ and $\supp(p)\subseteq E$. For $n\in S$ let $f_n$ be some function $f_n\colon\omega\to\{n\}\times\omega_1$ such that $\rng(f_n)\cap E=\varnothing$ for all $n$. Let $\dot a$ be the name $\tup{\dot f_n\mid n\in S}^\bullet$ which is a name for an element of $A_S$.

If $p\forces\dot a\notin\rng(\dot F)$, then $\dot F$ is not surjective. Otherwise, we may assume $p\forces\dot F$ is surjective, and we can extend $p$ to some $q$ such that there is some $\dot b$ which is a $\bullet$-name appearing in  $\dot A_T$, for some $T\in\cJ$, and $q\forces\dot F(\dot b)=\dot a$.

As the usual argument goes now, pick some $n\in S\setminus T$, and some $\alpha<\beta<\omega_1$ for which the following hold:
\begin{enumerate}
\item $f_n(m)=\alpha$ for some $m$,
\item $\tup{n,\alpha},\tup{n,\beta}\notin\supp(q)$.
\end{enumerate}
Then the permutation $\pi$ which acts only on the $n$th copy of $\omega_1$ and switches $\alpha$ with $\beta$ satisfies that:
\begin{enumerate}
\item $\pi\in\fix(E)$, and therefore $\pi\dot F=\dot F$ and $\pi p=p$.
\item $\pi\dot A_T=\dot A_T$, and in particular $\pi\dot b=\dot b$.
\item $\pi\dot x_{f_n}\neq\dot x_{f_n}$, and therefore $\pi\dot a\neq\dot a$.
\item And most importantly, $\pi q=q$.
\end{enumerate}
Therefore $\pi q=q\forces\pi\dot F(\pi\dot b)=\pi\dot a\neq\dot a=\dot F(\dot b)=\pi\dot F(\pi\dot b)$. This is of course impossible.

When dealing with $S$ or $\cJ$ which are not in $V$ we need to extend our conditions and $E$ to also preserve the relevant names for these sets, and extend $q$ to decide at least one natural number such that $n\in S\setminus T$, and the function $f_n$ (which is still in the ground model even when $S$ is not). As the conditions in the Cohen forcing are all finitary, this does not change the core of the above argument.
\end{proof}

This completes the proof of a slightly more general theorem than Krivine's, as we embed an even larger Boolean algebra. We observe that $\omega_1$ can be replaced by any $\kappa^+$ to preserve $\DC_\kappa$ as well, just as before.

\begin{theorem}
Assume that $V\models\ZFC+\GCH$. Let $\kappa$ be an infinite cardinal, then there is a cofinality-preserving symmetric extension $\cK_2(\kappa)$ in which the following statements hold:
\begin{enumerate}
\item $\ZF+\DC_\kappa+\AC_\WO$,
\item There is an embedding, $S\mapsto A_S$, from $\power(\omega)$ into $\power(\RR)$ with the properties that:
\begin{enumerate}
\item $A_\varnothing=\{0\}$, $A_S$ is uncountable for all $S\neq\varnothing$, and $|\RR|\leq^*|A_\omega|$,
\item $A_{S\cap T}=A_S\cap A_T$,
\item $|A_{S\cup T}|=|A_S\times A_T|$, in particular $|A_S\times A_S|=|A_S|$ for all $S$.
\item $|A_S|\leq^*|\bigcup_{T\in\cJ}A_T|$ if and only if $S\subseteq T$ for some $T\in\cJ$, for any $\cJ\subseteq\power(\omega)$.
\end{enumerate}
\end{enumerate}
\end{theorem}
We did not referred to the $\gimel$ function in this construction. Krivine's embedding utilizes a type of embedding from $\gimel 2$ to $\gimel_i\power(\kappa)$, where $\kappa$ is collapsed to be countable. Firstly, we did not collapse any cardinals. Moreover, choosing $\kappa=\omega$ to begin with, the forcing to collapse $\kappa$ is just adding a Cohen real. This seems to be somewhat similar to our approach, although several obvious differences still exist (e.g., we add $\omega_1$ Cohen reals). If, however, we will try to replicate the approach of the previous model, then it seems to hint towards defining for every sequence of interpretations of ``a name in $\{0,1\}$'' a sequence of real numbers which behaves like a subset of $\kappa$, which is then coded as a sequence of ``possible subsets of $\kappa$'' and by countability becomes a sequence of reals.

\subsection{The failing Model III: Oddly ordered set}\label{subsect:failure}
Finally, we discuss our failed attempt to construct the third model of Krivine, which is the second model mentioned in \S5.1 of \cite{Krivine:RIII}.
\begin{Kthm}[Krivine, \S5.1 in \cite{Krivine:RIII}]
It is consistent with $\ZF+\DC$ that there is $X\subseteq\RR$ s.t.
\begin{enumerate}
\item $|X|=|X\times X|$,
\item $\aleph_1\nleq^*|X|$,
\item $X$ admits a linear order where every proper initial segment is countable,
\item $|\RR|\leq^*|X\times\omega_1|\leq|\RR|$.
\end{enumerate}
\end{Kthm}
It seems quite clear that by taking $X=\omega_1\times\QQ$, with the lexicographic ordering, we obtain a linear ordering where every proper initial segment is countable, and in fact isomorphic to any other proper initial segment, except the empty set. Indeed, this is the model given in the multiplicative sequence part of this very paper. The model obtained there satisfies $\ZF+\DC+\AC_\WO$, $X$ has a linear ordering, $\prec$, where every proper initial segment is countable, and $\aleph_1\nleq^*|X|$.

As a consequence of \autoref{thm:choice-lemma} $\AC_X$ holds.\footnote{This, as Krivine informed us, holds in the realizability model as well. As is the case for $\AC_\WO$.} Therefore we can choose uniformly an enumeration for each $X\restriction x=\{y\in X\mid y\prec x\}$, and map $\tup{x,\alpha}$ to the interpretation of the $\alpha$th canonical name of a Cohen real in $\Add(\omega,\QQ)$, to its interpretation using $X\restriction x$ as a generic filter. By standard arguments it follows that this map is surjective.

This seems like we are done, but we are also required $|X|=|X\times X|$, which is blatantly false. In fact, this was part of the crux of the construction of the multiplicative sequence, since $A_2$ was a copy of $X$ itself. There are two immediate approaches to correct for this problem.

The first option is to preserve more information. Namely, fix a bijection of $\omega_1\times\QQ$ with its square, and ensure that the automorphisms preserve that bijection as well. The second option is to replace $X$ with some set defined from it, e.g.\ $X^{<\omega}$ or $X^{\omega,V}$. 

The first approach seems to require a refined model theoretic analysis which depends on the bijection $F$, since the main point of the construction is that we need to ensure that any point can be ``moved up'' arbitrarily high using a permutation. This is easy with arbitrary order automorphisms, but adding a bijection adds a lot more constraints which may also depend on the bijection.

The second approach fails because there is either no obvious bijection with the square\footnote{Recall that the standard argument for $|X^{<\omega}\times X^{<\omega}|=|X^{<\omega}|$ involves splitting $X$ into two disjoint parts, each equipotent with $X$. This seems to require more choice than we can afford.} or we somehow code a surjection onto $\omega_1$.

To make matters worse, trying to understand the $\gimel$ operator as a collection of possible names for a ground model set is not going to work here either, since $\gimel 2$ is a Boolean algebra with four elements in Krivine's model. There is a silver lining, though, to the finiteness of $\gimel 2$: after the release of the first draft of this paper Krivine has announced that over a model where $\gimel 2$ is finite we can in fact force the axiom of choice. This is interesting for us: recently Toshimichi Usuba proved (see Corollary~12 in \cite{Usuba:2019}) that if $M$ is a model of $\ZF$ and we can force the axiom of choice over $M$ (with a set forcing, of course), then $M$ is a symmetric extension of a definable inner model of $\ZFC$. Of course, there is no guarantee in the realizability case that this inner model is indeed the elementary extension of the ground model, nor we can pinpoint the symmetric extensions in full. Nevertheless, it shows that Krivine's consistency result can be obtained via classical methods.

Our result, while not quite that of Krivine, as we omit the first property of $X$, can be phrased as follows.
\begin{theorem}
Assume that $V\models\ZFC+\CH$. Then there is a cofinality preserving symmetric extension satisfying $\ZF+\DC+\AC_\WO$ in which there is a set $X\subseteq\RR$ such that:
\begin{enumerate}
\item $\aleph_1\nleq^*|X|$,
\item $X$ admits a linear ordering where every proper initial segment is countable,
\item $|\RR|\leq^*|X\times\omega_1|\leq|\RR|$, and
\item $\AC_X$ holds.
\end{enumerate}
\end{theorem}
\section{Concluding remarks}
It is always exciting to see new techniques for producing results in set theory, even if the results are old and known. 

For consistency results related to the axiom of choice we have symmetric extensions, relative constructibility, and the lesser-known method of forcing over models with atoms (see \cite{BlassScedrov:1989} for details). These are all tightly related to one another. For example, the method of symmetric extensions can be presented as we did here, or by cleverly choosing a set $A$ and considering $\HOD(V\cup A)^{V[G]}$, as was shown by Serge Grigorieff in \cite{Grigorieff:1975}.

We hope that this paper will motivate others to investigate the connections between realizability models and symmetric extensions. We suggest that as a complement to this paper, some of the famous results should be reproved using realizability. Since we need to preserve $\ZF+\DC$, Solovay's model-style constructions (e.g., preservation of large cardinal properties at $\omega_1$) seem like a good start.

One last point of interest to those coming from realizability would be to look at pre-Shoenfield independence results related to the axiom of choice. Cohen's original definition of forcing has a more intuitionistic flavor, and these were not always presented as a group acting on a forcing, but rather identified a family of names and used them to define a model. 

The following list is a list of questions we believe are important for understanding the connection between realizability models and symmetric extensions.
\begin{enumerate}
\item What kind of new reals are added? Are they all Cohen, for example, over the elementary extension of the ground model? Are they even generic over this copy to begin with?
\item Can we understand the $\gimel$ function in terms of names being interpreted by some set of canonical reals? Is it somehow related to Boolean-valued reduced powers? In a discussion with Krivine recently he suggested thinking about $\gimel 2$ as a measurement of how far away the realizability model is from a forcing extension (which is a trivial realizability construction). Grigorieff in \cite{Grigorieff:1975} showed that if $V\subseteq M\subseteq V[G]$ and $M$ is a symmetric extension, then there is a homogeneous Boolean algebra in $M$ such that $G$ is the generic object for it. In some sense, this Boolean algebra can be also seen as some measurement of how far we might be from $V[G]$, and the two objects might be somehow related. Unfortunately, this approach will not help to resolve the case of a finite $\gimel 2$ as in the third model.
\item In light of Krivine's newly announced results that the axiom of choice can be forced over models where $\gimel 2$ is finite, what can we say about the $\ZFC$ ground model inside those symmetric extensions? Is it the elementary extension of the ground model? In that case, can we pinpoint the symmetric system?
\item If we are only interested in the structure of the real numbers, can we reduce \textit{those results} to symmetric extensions as we did above?
\end{enumerate}  
\bibliographystyle{amsplain}
\providecommand{\bysame}{\leavevmode\hbox to3em{\hrulefill}\thinspace}
\providecommand{\MR}{\relax\ifhmode\unskip\space\fi MR }
\providecommand{\MRhref}[2]{%
  \href{http://www.ams.org/mathscinet-getitem?mr=#1}{#2}
}
\providecommand{\href}[2]{#2}

\end{document}